\documentclass[11pt]{amsart}
\usepackage{amssymb}
\usepackage{amsmath}
\usepackage{graphicx}
\usepackage{hyperref}
\usepackage{xcolor}
\usepackage{soul}
\usepackage{tikz}

\newtheorem{thm}{Theorem}[section]
\newtheorem{lem}[thm]{Lemma}
\newtheorem{fact}[thm]{Fact}

\newtheorem{Q}[thm]{Question}
\newtheorem{Def}[thm]{Definition}
\newtheorem{prop}[thm]{Proposition}
\newtheorem{rem}[thm]{Remark}

\newcommand{\bdfn}{\begin{Def} \rm}
\newcommand{\edfn}{\end{Def}}

\newcommand{\cross}{\mathbin{\tikz [x=1.4ex,y=1.4ex,line width=.2ex] \draw (0,0)--(1,1) (0,1)--(1,0);}}

\newcommand{\beqa}{\begin{eqnarray*}}
\newcommand{\eeqa}{\end{eqnarray*}}

\newcounter{cnt1}
\newcounter{cnt2}
\newcounter{cnt3}
\newcounter{cnt4}
\newcommand{\blr}{\begin{list}{$($\roman{cnt1}$)$} {\usecounter{cnt1}
 \setlength{\topsep}{0pt} \setlength{\itemsep}{0pt}}}
\newcommand{\blR}{\begin{list}{\Roman{cnt4}.\ } {\usecounter{cnt4}
 \setlength{\topsep}{0pt} \setlength{\itemsep}{0pt}}}
\newcommand{\bla}{\begin{list}{$(\alph{cnt2})$} {\usecounter{cnt2}
 \setlength{\topsep}{0pt} \setlength{\itemsep}{0pt}}}
\newcommand{\bln}{\begin{list}{$($\arabic{cnt3}$)$} {\usecounter{cnt3}
 \setlength{\topsep}{0pt} \setlength{\itemsep}{0pt}}}
\newcommand{\el}{\end{list}}

\begin{document}

\title[\tiny{A study on coreflexive Banach spaces}]{A study on coreflexive Banach spaces}

\author[Dwivedi]{Saurabh Dwivedi}

\address{Shiv Nadar Institution of Eminence. Gautam Buddha Nagar Delhi NCR, Uttar Pradesh-201314, India}

\email{sd605@snu.edu.in, saurabhdewedi876@gmail.com}

\subjclass[2020]{Primary 46A22, 46B10, 46B25; Secondary 46B20, 46B22. \hfill
\textbf{\today}}

\begin{abstract}
In this paper, we study non-reflexive Banach spaces $X$ for which the quotient space $X^{**}/X$ is reflexive. Such spaces were first introduced by James R.~Clark in \cite{CJ}, where they were called coreflexive spaces. In Theorem~\ref{T25}, we show that a space $X$ is coreflexive if and only if every separable subspace $Y\subseteq X$ is coreflexive, provided that $X$ is w$^*$-sequently dense in its bidual $X^{**}$. We show that coreflexive spaces are stable under $\ell^{p}$-sum for $1<p<\infty$. In Theorem~\ref{T210}, we show that if $X$ is a coreflexive space such that $X^{**}/X$ is separable, then the space of Bochner $p$-integrable functions, $L^{p}(\mu,X)$ is coreflexive for $1<p<\infty$. We conclude by providing an alternative proof of the fact, in a quasi-reflexive space $X$, w-PC's of the unit ball $X_{1}$ continues to have the same property in all the higher even-order dual unit balls of $X$.
\end{abstract}
\maketitle

\textbf{To appear in Proceedings - Mathematical Sciences.}
\section{Introduction}
Let $X$ be a Banach space and let $X^*$ denotes its dual space. Let $X_{1}$ denotes the closed unit ball and let $S(X)$ be the unit sphere of $X$. We denote the canonical embedding of $X$ in to its bidual $X^{**}$ by identifying $X$ with its image under the map $J_{X}:X\to X^{**}$, where $J_{X}$ is the canonical embedding. For clarity, we sometimes write $J_{X}(X)$ to explicitly refer to the embedded copy of $X$ in $X^{**}$. More generally, for $n\geq 1$, we use $X^{(n)}$ to denote the $n^{th}$-order dual space of $X$. In this context, the inclusion $X^{(n)}\subseteq X^{(n+2)}$ refers to the canonical embedding of $X^{(n)}$ in to $X^{(n+2)}$. We begin by defining a few generalised notions of reflexivity. We recall that a space $X$ is said to be \textbf{quasi-reflexive} if the quotient space $X^{**}/X$ is finite dimensional. A detailed study of quasi-reflexive spaces can be found in \cite{Ho} and \cite{CB}. Motivated by this, we recall from \cite{CJ} that a space $X$ is said to be \textbf{coreflexive} if the quotient space $X^{**}/X$ is reflexive. Clearly, every quasi-reflexive space is coreflexive. We recall that a space $X$ is said to be a \textbf{weakly compactly generated} (WCG) space if there exists a weakly compact subset $K\subseteq X$ such that $X=\overline{\textnormal{span}}(K)$ (norm closure). It is easy to see that both reflexive and separable spaces are weakly compactly generated. Throughout, we use the symbol $\cong$ to denote an isomorphism. The motivation of studying coreflexive spaces comes from the work of Davis, W. J., Figiel, T., Johnson, W. B.,  Pelczynski, A. (see \cite[Proposition~1]{TF}), where it was shown that given a weakly compactly generated space $X$, there exists a space $Z$ such that $Z^{**}/Z\cong X$. In particular, As $X$ is reflexive, it is weakly compactly generated, so there exists a space $Z$ such that $Z^{**}/Z\cong X$. Clearly, such a space $Z$ is a coreflexive space. These spaces occurs naturally and have many geometric properties similar to the reflexive spaces. For example, recall that a point $x^*\in X_{1}^*$ is said to be a \textbf{w$^*$-w} PC if the identity map $I:(X_{1}^{*},\textnormal{w}^*)\rightarrow(X_{1}^{*},\textnormal{w})$ is continuous at $x^*$. We say that a point $x^*\in X_{1}^*$ is w$^*$-w PC in all the higher odd-order dual unit balls of $X$, if for every odd $n\in \mathbb{N}$, the canonical image of $x^*$ in to $X_{1}^{(n)}$ is a w$^*$-w PC. In \cite [Proposition~5.3]{STT}, the authors have shown that if $X$ is coreflexive, then w$^*$-w PC's of $X_{1}^*$ continues to have the same property in all the higher odd-order dual unit balls of $X$. Let $\{X_{k}\}_{k\in \mathbb{N}}$ be a family of Banach spaces. For $1\leq p<\infty$, we define $\bigoplus_{{\ell^{p}}}X_{k}=\{(x_k): x_k\in X_k, \sum_{k\in \mathbb{N}}\|x_k\|^{p}<\infty\}$, equipped with the $\ell^{p}$-norm. We define $\bigoplus_{{\ell^{\infty}}}X_{k}=\{(x_k): x_k\in X_k, \textnormal{ sup}_{k\in \mathbb{N}}\|x_{k}\|<\infty\}$, equipped with the norm $\|(x_{k})\|_{{\infty}}=\textnormal{ sup}_{k\in \mathbb{N}}\|x_{k}\|$. Let $(\Omega, \Sigma, \mu)$ be a measure space, where $\mu$ is a positive measure. Let $L^p(\mu,X)$ denotes the Bochner space of $p$-integrable functions taking values in $X$, equipped with the norm $\|f\|_{p}=\left(\int_{\Omega}\|f(\omega)\|^{p}\right)^{\frac{1}{p}}d(\mu(\omega))$, for $f\in L^p(\mu,X)$. Further discussion on this topic can be found in \cite[Chapter~4]{DU}. 

The main results of this paper are summarized as follows. As a background fact, it is well known (see Proposition~\ref{T22}) that if a space $X$ is coreflexive, then every closed subspace $Y$ of $X$ is coreflexive. As a consequence of the \textbf{Eberlein–Šmul'yan} theorem (See \cite[Pg.~149]{Ho}), we know that a space $X$ is reflexive if and only if every closed separable subspace of $X$ is reflexive. An analogous result to this is established in Theorem~\ref{T25}, where it is shown that a space $X$ is coreflexive if and only if every separable subspace $Y\subseteq X$ is coreflexive, provided that $X$ is w$^*$-sequently dense in its bidual $X^{**}$. In Theorem~\ref{T27}, we show the stability of coreflexive spaces under $\ell^{p}$-sum for $1<p<\infty$. 

In section~3, we start with Theorem~\ref{P32}, there we demonstrates that if $X$ is a coreflexive space such that the quotient space $X^{**}/X$ is separable, then $X$ and all of its dual spaces satisfy the Radon–Nikodým property (R.N.P.). We refer the reader to \cite[Chapter~3]{DU} for a detailed study on the R.N.P. In Theorem~\ref{T210}, we show that if $X$ is a coreflexive space such that $X^{**}/X$ is separable, then $L^{p}(\mu,X)$ is coreflexive for $1<p<\infty$. We also provide a measure theoretic approach to show the coreflexivity of the space $L^{p}(\mu,X)$ in a special case using a different set of hypothesis on $X$. Recall that a point $x\in X_{1}$ is said to be \textbf{w-PC}   if the identity map $I:(X_{1},\textnormal{w})\rightarrow(X_{1},\|.\|)$ is continuous at $x$. Let $x\in X_{1}$ is a w-PC. We say that the point $x\in X_{1}$ is w-PC in all the higher even-order dual unit balls of $X$, if for every even $n\in \mathbb{N}$, the canonical image of $x$ in to $X_{1}^{(n)}$ is a w-PC. We conclude this note by providing an alternative proof of Fact~\ref{F38}, which states that in a quasi-reflexive Banach space $X$, w-PC's of $X_{1}$ retain the same property in all the higher even-order dual unit balls of $X$.
\section{Basic results on coreflexive spaces}
For a closed subspace $Y$ of $X$, it is known (See \cite[Exercise~1]{Ho}) that $Y^{**}/Y$ can be identified with a closed subspace of $X^{**}/X$ . For the sake of completeness, we include a proof of this fact below. Specifically in the following proposition,  we denote by $J_{X}(X)$ and $J_{Y}(Y)$ to denote the canonical embeddings of $X$ in to $X^{**}$ and $Y$ into $Y^{**}$, respectively, in order to avoid ambiguity. For a closed subspace $Y$ of $X$, we use $Q_{Y}$ to denote the quotient map from $X$ onto the quotient space $X/Y$.

\begin{prop}\label{T22}
 Let $Y$ be a closed subspace of $X$. Then $(J_{X}(X)+Y^{\perp\perp})/{J_{X}(X)}$ is a closed subspace of ${X^{**}}/{J_X(X)}$ and ${Y^{**}}/{J_Y(Y)}$ is isomorphic to $(J_{X}(X)+Y^{\perp\perp})/{J_{X}(X)}$
\end{prop}
\begin{proof}
	Let $X$ be a coreflexive space and let $Q_{Y}:X\to X/{Y}$ be the quotient map. It is easy to see that $Q_{Y}^{**}\circ J_{X}=J_{X/Y}\circ Q_{Y}$. So we have, $Q_{Y}^{{**}{-1}}\circ J_{X/Y}\left(X/{Y}\right)=J_{X}(X)+Y^{\perp\perp}$. Thus $J_{X}(X)+Y^{\perp\perp}$ is a closed subspace of $X^{**}$ and hence $(J_{X}(X)+Y^{\perp\perp})/{J_{X}(X)}$ is a closed subspace of ${X^{**}}/{J_X(X)}$.

   Next we show that ${Y^{**}}/{J_Y(Y)}$ is isomorphic to $(J_{X}(X)+Y^{\perp\perp})/{J_{X}(X)}$. Let $R_{Y}:X^*\to Y^*$ be the restriction map with ker($R_{Y})=Y^{\perp}$. It is easy to see that $R_{Y}^*$ is an isometry from $Y^{**}$ onto $Y^{\perp\perp}$. Clearly, $R_{Y}^*\circ J_{Y}=J_{X}|_{Y}$. Now we consider the map $Q_{J_{X}(X)}\circ R_{Y}^*:Y^{**}\to (J_{X}(X)+Y^{\perp\perp})/{J_{X}(X)}$. First we show that ker$(Q_{J_{X}(X)}\circ R_{Y}^*)=J_{Y}(Y)$. Let $R_{Y}^*(y^{**})\in J_{X}(X)$, say $R_{Y}^*(y^{**})=J_{X}(x_{0})$ for some $x_{0}\in X$ and $\{y_{\alpha}\}_{\alpha\in I}\subseteq Y$ be a net such that $J_{Y}(y_{\alpha})\to y^{**}$ in the w$^*$-topology of $Y^{**}$. Since the adjoint of a bounded linear map is always w$^*$-continuous (See \cite[Pg.~121]{Ho}), we have that $R_{Y}^{*}$ is w$^*$-continuous. Consequently, $J_{X}(y_{\alpha})=R_{Y}^{*}(J_{Y}(y_{\alpha}))\to R_{Y}^*(y^{**})=J_{X}(x_{0})$ in the w$^*$-topology of $X^{**}$. So we get that $y_{\alpha}\to x_{0}$ in the w-topology. Thus $x_{0}\in Y$. It gives that $R_{Y}^*(y^{**})=J_{X}(x_{0})=R_{Y}^*(J_{Y}(x_{0}))$. Since $R_{Y}^*$ is an onto isometry, we have $y^{**}=J_{Y}(x_{0})$. Hence ker$(Q_{J_{X}(X)}\circ R_{Y}^*)\subseteq J_{Y}(Y)$ The other way inclusion is easy to check. Now observe that the map $Q_{J_{X}(X)}\circ R_{Y}^*$ is surjective, as $R_{Y}^*$ is surjective. We have that $Y^{**}/J_{Y}(Y)$ is isomorphic to $(J_{X}(X)+Y^{\perp\perp})/{J_{X}(X)}$.
\end{proof}

Let $Y\subseteq Z\subseteq X^*$ be two subsets. We say $Y$ is \textbf{w$^*$-dense} in $Z$, if for each point $z^*\in Z$, there exists a net $\{y_{\alpha}^*\}_{\alpha\in J}\subseteq Y$ such that $y_{\alpha}^*\to z^*$ in the w$^*$-topology. While,  the subset $Y$ is said to be \textbf{w$^*$-sequently dense} in $Z$, if for each point $z^*\in Z$, there exists a sequence $\{y_{n}^*\}_{n\in \mathbb{N}}\subseteq Y$ such that $y_{n}^*\to z^*$ in the w$^*$-topology. We denote by $\overline{Y}^{\textnormal{w}^*}$, the closure of $Y$ in the w$^*$-topology. We begin this note by recalling the \textbf{Goldstine-Weston density lemma} from \cite{Ho}.

\begin{lem}\cite[Pg.~126]{Ho}\label{L21}
	Let $X$ be a Banach space. Then $X_{1}$ is \textnormal{w}$^*$-dense in $X_{1}^{**}$.
\end{lem}

In the following lemma, we assume that $X$ is w$^*$-sequently dense in $X^{**}$. This provides a link between separable subspaces of $X^{**}$ and separable subspaces of $X$. Let $Y$ be a closed subspace of $X$. We recall that the annihilator of $Y$, denoted by $Y^{\perp}$, is defined as $\{x^*\in X^*:x^*(y)=0 \text{ for all }y\in Y\}$.
\begin{lem}\label{L23}
	Let $X$ be such that it is \textnormal{w}$^*$-sequently dense in $X^{**}$. For a closed separable subspace $Z_{0}\subseteq X^{**}$, there exists a closed separable subspace $Z_{}\subseteq X$ such that $Z_{0}\subseteq Z_{}^{\perp\perp}$.
\end{lem}
	\begin{proof}
    Let $Z_{0}\subseteq X^{**}$ be a separable subspace and let $\{x_{n}^{**}:n\in\mathbb{N}\}\subseteq Z_{0}$ be a countable dense susbset. Fix $n\in\mathbb{N}$, we have a sequence $(x_{n}(k))_{k\geq 1}\subseteq X$ such that $x_{n}(k)\to x_{n}^{**}$ as $k\to \infty$ in the w$^*$-topology. Let $Z_{}=\overline{\text{span}}\{x_{n}(k):n, k\in \mathbb{N}\}$ (norm closure). Clearly, $Z$ is a closed separable subspace of $X$. Since $Z^{\perp\perp}=\overline{Z}^{\textnormal{w}^*}$ (see \cite[Pg.~124]{Ho}), we have that $Z_{0}\subseteq Z_{}^{\perp\perp}$. This completes the proof.
\end{proof}

\begin{rem}
	Note that the assumption that $X$ is \textnormal{w}$^*$-sequently dense in $X^{**}$ is always satisfied when $X^*$ is separable. Indeed, the closed unit ball $X_{1}^{**}$ being a \textnormal{w}$^*$-compact subset of $X^{**}$ is metrizable in the \textnormal{w}$^*$-topology when $X^*$ is separable (See \cite[Theorem~3.16]{RW}).
\end{rem}

As we have discussed in the introduction that reflexivity of a space is separably determined property, in the below, we will discuss a similar characterization of coreflexive spaces. We use Lemma~\ref{L23} to establish the next theorem. Throughout, we use $Q_{X}$ to denote the quotient map from $X^{**}$ onto the quotient space $X^{**}/X$, where $X\subseteq X^{**}$ is identified with its canonical image in $X^{**}$.
\begin{thm}\label{T25}
    Let $X$ be such that it is \textnormal{w}$^*$-sequently dense in $X^{**}$. If every closed separable subspace $Y\subseteq X$ is coreflexive, then $X$ is a coreflexive space.
\end{thm}
\begin{proof}
   Suppose every closed separable subspace $Y\subseteq X$ is coreflexive. By the virtue of the \textbf{Eberlein–Šmul'yan} theorem (See \cite[Pg.~149]{Ho}), it is enough to show that every bounded sequence in the quotient space $X^{**}/X$ has a weakly convergent subsequence. Suppose that there exists a bounded sequence $(Q_{X}(x_{n}^{**}))_{n\in \mathbb{N}}\subseteq X^{**}/X$ with no weakly convergent subsequence. Let $Z_{0}= \overline{span}\{x_{n}^{**}:n\in \mathbb{N}\}$ (norm closure), a separable subspace of $X^{**}$. By the hypothesis and Lemma~\ref{L23}, we have a separable subspace $Z_{}\subseteq X$ such that $Z_{0}\subseteq Z^{\perp\perp}=\overline{Z}^{\textnormal{w}^*}$. From proposition~\ref{T22}, we have that $(Z^{\perp\perp}+X)\subseteq X^{**}$ is a closed subspace and $(Z^{\perp\perp}+X)/X\cong Z^{**}/Z$. Since $Z\subseteq X$ is separable, we have that the quotient space $(Z^{\perp\perp}+X)/X$ is reflexive. Note that $(Q_{X}(x_{n}^{**}))_{n\in \mathbb{N}}\subseteq (Z^{\perp\perp}+X)/X$. Since $(Q_{X}(x_{n}^{**}))_{n\in \mathbb{N}}$ has no weakly convergent subsequence in $X^{**}/X$, we have that $(Q_{X}(x_{n}^{**}))_{n\in \mathbb{N}}$ has no weakly convergent subsequence in $(Z^{\perp\perp}+X)/X$ as it is a closed subspace of $X^{**}/X$ . This is a contradiction to the fact that $(Z^{\perp\perp}+X)/X$ is reflexive. Hence the quotient space $X^{**}/X$ is reflexive, therefore $X$ is coreflexive. This completes the proof.
\end{proof}
Let us recall the canonical decomposition of the third dual $X^{***}$. Consider the canonical embedding $J_{X}:X\to X^{**}$ and its adjoint $J_{X}^*:X^{***}\to X^*$. Clearly, ker$(J_{X}^*)=X^{\perp}\subseteq X^{***}$. It is esay to see that the map $P=J_{X^*}\circ J^{*}_{X}:X^{***}\to X^{***}$ is a projection such that ker$(P)=X^{\perp}$ and range$(P)=X^*$. This gives that $X^{***}=X^*\oplus X^{\perp}$. We use this fact to establish the next proposition. Which connects the coreflexivity of a space $X$ to the reflexivity of a subspace of $X^{***}$.
\begin{prop}\label{P27}
	Let $X$ be a Banach space. Then $X$ is coreflexive if and only if $X^{\perp}\subseteq X^{***}$ is reflexive.
\end{prop}
\begin{proof}
	From the above discussion, we have the decomposition of $X^{***}$ as $X^{***}=X^*\oplus X^{\perp}$. This implies that the quotient space $X^{***}/X^*$ is isomorphic to $X^{\perp}$. It is shown (See \cite[Theorem~2.1]{CJ}) that a space $X$ is coreflexive if and only if $X^*$ is coreflexive, that is, $X^{**}/X$ is reflexive if and only if $X^{***}/X^*$ is reflexive. So we get that $X$ is coreflexive if and only if $X^{\perp}$ is a reflexive subspace of $X^{***}$.
\end{proof}
The next theorem establishes that the coreflexivity is stable under $\ell^{p}$-sum for each $1<p<\infty$. For $n\in \mathbb{N}$, we denote by $xe_{n}$, a sequence of elements in $X$ in which $n^{th}$-coordinate is $x$ and $0$ otherwise.

\begin{thm}\label{T27}
Let $\{X_{k}\}_{k\in \mathbb{N}}$ be a family of spaces such that each $X_{k}$ is coreflexive. Then for any $1< p<\infty$, the space $\bigoplus_{\ell^{p}}X_{k}$ is coreflexive.
\end{thm}
\begin{proof}
    From Proposition~\ref{P27}, we have that a space $X$ is coreflexive if and only if the closed subspace $X^{\perp}\subseteq X^{***}$ is reflexive. So it is enough to show that $\left(\bigoplus_{\ell^{p}}X_{k}\right)^{\perp}\subseteq \bigoplus_{\ell^{q}}X_{k}^{***}$ is reflexive. We claim that $\left(\bigoplus_{\ell^{p}}X_{k}\right)^{\perp}\subseteq \bigoplus_{\ell^{q}}X_{k}^{\perp}$. Let $(x_{k}^{***})_{k\in\mathbb{N}}\in \left(\bigoplus_{\ell^{p}}X_{k}\right)^{\perp}$. Fix $k_{0}\in \mathbb{N}$ and $x\in X_{k_{0}}$, we have $(xe_{k_{0}})\in \bigoplus_{\ell^{p}}X_{k}$. So we get $(x_{k}^{***})(xe_{k_{0}})=x_{k_{0}}^{***}(x)=0$. Since $x\in X_{k_{0}}$ is an arbitrary choice, we have $x_{k_{0}}^{***}\in X_{k_{0}}^{\perp}\subseteq X_{k_{0}}^{***}$. This implies $(x_{k}^{***})_{k\in \mathbb{N}}\in \bigoplus_{\ell^{q}}X_{k}^{\perp}$, as each component $x_{k}^{***}$ belongs to the corresponding subspace $X_{k}^{\perp}$. This establishes the claim. By the hypothesis, each $X_{k}^{\perp}$ is reflexive. So we have that $\bigoplus_{\ell^{q}}X_{k}^{\perp}$ is a reflexive space. Consequently, the subspace $\left(\bigoplus_{\ell^{p}}X_{k}\right)^{\perp}\subseteq \bigoplus_{\ell^{q}}X_{k}^{\perp}$ is reflexive. Hence $\bigoplus_{\ell^{p}}X_{k}$ is a coreflexive space. This completes the proof.
\end{proof}

\begin{rem}
    In particular, if we have a family of quasi-reflexive spaces, say $\{X_{k}\}_{k\in\mathbb{N}}$, then it follows from Theorem~\ref{T27} that the space $\bigoplus_{\ell^{p}}X_{k}$ is coreflexive for all $1< p<\infty$.
\end{rem}
\begin{rem}
	The proof given above does not rely on the countability of the indexing set. Hence, the conclusion of Theorem~\ref{T27} remains valid for an arbitrary family of Banach spaces.
\end{rem}
The following remark deals with the special cases when $p = 1$ or $p = \infty$.
\begin{rem}
	When the indexing set is infinite, the spaces $\bigoplus_{\ell^{1}} X_{k}$ and $\bigoplus_{\ell^{\infty}} X_{k}$ contain isometric copies of $\ell^{1}$ and $\ell^{\infty}$, respectively. It is known (see\cite[Theorem~2.3]{CJ}) that both $\ell^{1}$ and $\ell^{\infty}$ are not coreflexive. Hence, neither of $\bigoplus_{\ell^{1}} X_{k}$ and $\bigoplus_{\ell^{\infty}} X_{k}$ can be coreflexive.
\end{rem}
In the next proposition, we will show that the property that a space $X$ is w$^*$-sequently dense in to its bidual $X^{**}$, is stable under $\ell^{p}$-sum for $1<p<\infty$. 

\begin{prop}
	Let $\{X_{k}\}_{k\in \mathbb{N}}$ be a family of Banach spaces and let each $X_{k}$ is \textnormal{w$^*$}-sequently dense in $X_{k}^{**}$. Then for any $1< p<\infty$, the space $\bigoplus_{\ell^{p}}X_{k}$ is \textnormal{w}$^*$-sequently dense in $\bigoplus_{\ell^{p}}X_{k}^{**}$.
\end{prop}
\begin{proof}
	Let $(x_{k}^{**})\in \bigoplus_{\ell^{p}}X_{k}^{**}$. By the hypothesis, for each $k\in \mathbb{N}$, there exists a sequence $(x_{n}(k))_{n\in \mathbb{N}}\subseteq X_{k}$ such that $x_{n}(k)\to x_{k}^{**}$ as $n\to \infty$, in the \textnormal{w}$^*$-topology. Fix $n\in \mathbb{N}$ and define\\
\[
z_{k}(n) =
\begin{cases} 
    x_{n}(k), & \text{if } 1\leq k \leq n, \\
    0, & \text{if } k > n.
\end{cases}
\]
 We write $z(n)=(z_{k}(n))_{k\geq1}$. Thus we get a sequence $(z(n))_{n\in \mathbb{N}}\subseteq \bigoplus_{\ell^{p}}X_{k}$  such that for each $k\in\mathbb{N}$, $z_{k}(n)\to x_{k}^{**}$ as $n\to \infty$ in the w$^*$-topology. So we get that the sequence $(z(n))_{n\in \mathbb{N}}$ converges to $(x_{k}^{**})$ in the \textnormal{w}$^*$-topology. This completes the proof.
\end{proof}
\section{coreflexivity of Bochner $p$-integrable space}
This section is devoted to explore the coreflexivity of the space $L^{p}(\mu, X)$ for $1<p<\infty$. Let $X$ be a Banach space such that both $X^{*}$ and $X^{**}$ have the R.N.P. Recall that for $1 < p < \infty$, the dual space is given by $L^{p}(\mu, X)^{*} = L^{q}(\mu, X^{*})$, where $1 < q < \infty$ is such that $\tfrac{1}{p} + \tfrac{1}{q} = 1$ (see \cite[Pg.~98, Theorem~1]{DU}).
 Since $X^{**}$ also has the R.N.P., it follows that $$L^{p}(\mu, X)^{**}=L^{q}(\mu, X^{*})^*=L^{p}(\mu, X^{**}).$$ Therefore, in this setting, the natural inclusion $L^{p}(\mu, X)\subseteq L^{p}(\mu, X^{**})$ is precisely the canonical embedding of $L^{p}(\mu, X)$ into its bidual. In the above, the R.N.P. assumptions on $X^*$ and $X^{**}$ can be achieved by considering $X$, a coreflexive space, and $X^{**}/X$, a separable space. The idea of assuming that $X^{**}/X$ is a separable space comes from the work of M. Valdivia in \cite{VM}. There it is shown that a space $X$ with $X^{**}/X$ be a separable space can be written as $R\oplus S$ (topological direct sum), where $R\subseteq X$ is reflexive and $S\subseteq X$ is separable. We first utilise this fact to establish the following proposition.
\begin{thm}\label{P32}
	Let $X$ be a coreflexive space such that $X^{**}/X$ is separable. Then $X$ and all of its dual spaces satisfy the R.N.P. 
\end{thm}
\begin{proof}
	Since the quotient space $X^{**}/X$ is reflexive and separable, we have that its dual space $(X^{**}/X)^*\cong X^{\perp}$ is reflexive and separable. Using the canonical decomposition $X^{***}=X^*\oplus X^{\perp}$, we get that $X^{***}/X^*$ is also a separable reflexive space. Now we can apply Valdivia's result to the space $X^*$ to write $X^*=R\oplus S$, where $R\subseteq X^*$ is reflexive and $S\subseteq X^*$ is a separable space.
	
	We first show that $X^*$ has the R.N.P. by showing that both $R$ and $S$ have the R.N.P. Since $R$ is reflexive, so it is w$^*$-closed in $X^*$ (See \cite[Pg.~196]{Ho}). As every w$^*$-closed subspace of $X^*$ can be viewed as a annihilator of some closed subspace of $X$ (See \cite[Theorem~3.4]{RW}), we have that $R=Y^{\perp}$ for some closed subspace $Y\subseteq X$. So we get
	\[
	Y^*\cong X^*/Y^{\perp}=X^*/R\cong S.
	\]
	This implies that $S$ is a separable dual space. From \cite[Pg.~79, Theorem~1]{DU}, we have that $S$ satisfies the R.N.P. As $R$ is a reflexive Banach space, $R$ also has the R.N.P. (See \cite[Pg.~76, Corollary~13]{DU}).  Since $X^*$ is the sum of two subspaces, each of which satisfies the R.N.P., it follows that $X^*$ itself satisfies the R.N.P. (See \cite[Theorem~6.5b]{CA}). As $X^{***}/X^*$ is a separable reflexive space, a similar line of reasoning shows that $X^{**}$ has the R.N.P. By an inductive argument, one can show that all the higher order dual spaces of $X$ possess the R.N.P. Clearly, $X$ has the R.N.P. as $X\subseteq X^{**}$ is a closed subspace (See \cite[Pg.~81, Theorem~2]{DU}). This completes the proof.
\end{proof}
Now we recall a well known proposition from \cite{STT} for a better understanding of the quotient spaces of $L^{p}(\mu, X)$.

\begin{prop}\cite[Proposition~3.1]{STT}\label{P29}
	Let $(\Omega,\Sigma,\mu)$ be a probability space. Let $Y$ be a closed subspace of $X$. Fix $1\leq p<\infty$. Then $L^{p}(\mu, X)/L^{p}(\mu, Y)$ is isometric to $L^{p}(\mu, X/Y)$.
\end{prop}
We use Theorem~\ref{P32} and Proposition~\ref{P29} to explore the coreflexivity of the space $L^{p}(\mu, X)$ for $1<p<\infty$. We begin by recalling the fact that the space $L^{p}(\mu, X)$ is reflexive, whenever $X$ is reflexive (See \cite[Pg.~100, Corollary~2]{DU}).
\begin{thm}\label{T210}
     Let $(\Omega,\Sigma,\mu)$ be a probability space. Let $X$ be such that $X^{**}/X$ is a separable space. For $1< p<\infty$, if $X$ is coreflexive, then $L^{p}(\mu, X)$ is coreflexive.
\end{thm}
\begin{proof}
    Since the quotient space $X^{**}/X$ is separable and reflexive, by Theorem~\ref{P32}, we have that $X^*$ and $X^{**}$ have the R.N.P. So we can identify the dual space $L^{p}(\mu, X)^{**}$ as the space $L^{p}(\mu, X^{**})$. Now we use Proposition~\ref{P29} to get that the quotient space $L^{p}(\mu, X^{**})/L^{p}(\mu, X)$ is isometrically isomorphic to $L^{p}(\mu, X^{**}/X)$. By the hypothesis, the space $X^{**}/X$ is reflexive. Thus $L^{p}(\mu, X^{**}/X)$ is reflexive. Also, we have
	\[L^{p}(\mu, X)^{**}/L^{p}(\mu, X)\cong L^{p}(\mu, X^{**})/L^{p}(\mu, X)\cong L^{p}(\mu, X^{**}/X)\] 
	Hence $L^{p}(\mu, X)^{**}/L^{p}(\mu, X)$ is a reflexive space, that is, $L^{p}(\mu, X)$ is coreflexive. This completes the proof. 
\end{proof}
\begin{rem}
	As mentioned during the proof of Theorem~\ref{P32} that if $X$ is coreflexive and $X^{**}/X$ is a separable space, the same hypothesis is also satisfied by its dual space $X^*$, so we have that $L^{p}(\mu, X^*)$ is also a coreflexive space for all $1<p<\infty$.
\end{rem}
\begin{rem}
In the proof of Theorem~\ref{T210}, the assumption that $X^{**}/X$ is separable can be replaced by directly assuming the Radon–Nikodým property (R.N.P.) for both $X^*$ and $X^{**}$. 	
\end{rem}
In \cite{STT}, the authors make use of a topological selection theorem, namely the \textbf{Michael's Selection} theorem, to prove Proposition~\ref{P29}, which is later utilized in the derivation of Theorem~\ref{T210}. In the next theorem, we examine the coreflexivity of the space $L^{p}([0,1], X)$ using the \textbf{Von-Neumann's selection} theorem and a different set of hypothesis on $X$. which provides a measure theoretic route to the result. This is in contrast to the proof given in Theorem~\ref{T210}. Let $([0,1],\mathcal{A},\mu)$ denotes the Lebesgue measure space, where $\mathcal{A}$ is the Lebesgue $\sigma$-field and $\mu$ is the Lebesgue measure. Let $V$ be a Polish space (separable completely metrizable topological space). We denote by $\mathcal{B}_V$, the $\sigma$-algebra generated by the topology of $V$ and by $\mathcal{A}\otimes \mathcal{B}_V$, the product $\sigma$-algebra on $[0,1]\cross V$ (see \cite[p. 83, 87]{S} for details). We proceed by recalling the Von-Neumann's selection theorem from \cite{S} in our setting.
\begin{thm}[Von-Neumann]\cite[Corollary~5.5.8]{S}\label{VNT}
   If $B\in \mathcal{A}\otimes \mathcal{B}_V$, then there exists a $\mu$-measurable selection $h:[0,1]\to V$ such that $(t,h(t))\in B$ for all $t\in [0,1]$.
\end{thm}
We now present a preparatory lemma, which will be used subsequently in the proof of Theorem~\ref{T212}.
\begin{lem}\label{L36}
	Let $X$ be a space such that it has the R.N.P. and $X^*$ is separable. If $X$ is coreflexive, then all the dual spaces of $X$ have the R.N.P.
\end{lem}
\begin{proof}
	Let $X$ be as given in the theorem. By the hypothesis, we have that $X^{**}/X$ is reflexive, hence, satisfies the R.N.P. (See \cite[Pg.~76, Corollary~13]{DU}). So we get that both $X^{**}/X$ and $X$ satisfy the R.N.P. It is shown in \cite[Theorem~6.5b]{CA} that if $Y$ and $X/Y$ have the R.N.P. for some closed subspace $Y\subseteq X$, then $X$ has the R.N.P. Applying this result to $X^{**}$ and $X$, we get that $X^{**}$ satisfies the R.N.P. As every separable dual space satisfies the R.N.P. (See\cite[Pg.~79, Theorem~1]{DU}), we have that $X^{*}$ has the R.N.P. 
	
	To show $X^{***}$ has the R.N.P. We first observe that $X^{**}$ is a separable space as it satisfies the R.N.P. and $X^*$ is separable. It is known (See \cite[Theorem~2.1]{CJ}) that $X^*$ is coreflexive if and only if $X$ is coreflexive. So we get that $X^*$ satisfies the hypothesis of the theorem. We can now apply the above argument to $X^*$ to conclude that $X^{***}$ is separable and satisfies the R.N.P. Inductively, we get that all the dual spaces of $X$ has the R.N.P. Additionally, all the dual spaces of $X$ are separable as in each step the dual space has the R.N.P. and the predual space is separable. 
\end{proof}
We now present the main theorem which will utilise Lemma~\ref{L36} to demonstrate the coreflexivity of the space $L^{p}([0,1],X)$ for $1<p<\infty$ using Von-Neumann's selection theorem.
Let $K$ be a Polish space (separable completely metrizable topological space) and let $T:[0,1]\to 2^{K}$ be a set-valued map. We define the graph $G(T)=\{(t,k)\in [0,1]\cross K: k\in T(t)\}$ of the map $T$. Given a set $E$, we denote by $E^{c}$, the set theoretic complement of $E$.
\begin{thm}\label{T212}
    Let $X$ be such that it has the R.N.P. and $X^*$ is separable. If $X$ is coreflexive, then $L^{p}([0,1], X)$ is coreflexive for all $1<p<\infty$. 
\end{thm}
\begin{proof}
    Since the Lebesgue $\sigma$-field is complete, by discarding null sets if necessary, we consider all the functions are everywhere defined. Recall from the Proposition~\ref{P27} that a space $X$ is coreflexive if and only if the subspace $X^{\perp}\subseteq X^{***}$ is reflexive. It is enough to show that $L^{p}([0,1], X)^{\perp}\subseteq L^{p}([0,1], X)^{***}$ is reflexive. By the hypothesis and Lemma~\ref{L36}, we have that all the duals of $X$ satisfy the R.N.P. So we can identify the space $L^{p}([0,1], X)^{***}$ as the space $L^{q}([0,1], X^{***})$. We claim that $L^{p}([0,1], X)^{\perp}\subseteq L^{q}([0,1], X^{\perp})\subseteq L^{q}([0,1], X^{***})$. Suppose $f\in L^{p}([0,1], X)^{\perp}$. We establish the claim by showing that the set $E=\{t\in [0,1]:f(t)\notin X^{\perp}\}\subseteq [0,1]$ is a null set. 
	
	First we show that the set $E$ is measurable. Since $X^*$ is a separable space, we have that $X$ is separable, and hence, $S(X)$ is separable. Let $\{x_{n}:n\in \mathbb{N}\}\subseteq S(X)$ be a countable dense subset. Let $A=\{x^{***}\in X^{***}:x^{***}\notin X^{\perp}\}$. For each $n\in \mathbb{N}$, we define $A_{n}=\{x^{***}\in X^{***}:x^{***}(x_{n})>0\}$. By Hahn-Banach theorem, we have have that $A_{n}$ is non-empty for each $n\in \mathbb{N}$. Also, each $A_{n}\subseteq X^{***}$ is a measurable subset as the evaluation at $x_{n}$ is a continuous linear function on $X^{***}$. It is easy to check that $A=\bigcup_{n=1}^{\infty}A_{n}$. So we get that $A\subseteq X^{***}$ is measurable. Since the function $f:[0,1]\to X^{***}$ is measurable, we have that $f^{-1}(A)$, that is, $E$ is a measurable subset of $[0,1]$. Now suppose that $\mu(E)>0$. This implies for each $t\in E$ there exists $x_{t}\in S(X)$ such that $f(t)(x_{t})>0$. For a fix point $x_{0}\in S(X)$, we define a set valued map $F:[0,1]\to 2^{S(X)}$ by
    \[
    F(t)=
    \begin{cases}
       \{x_0\}, & t \in E^{c}, \\
       \{x\in S(X): f(t)(x)>0\}, & t \in E.
    \end{cases}
    \]

The graph of $F$ is given by $G(F) = \{ (t, x)\in[0,1]\cross S(X)  :x\in F(t) \}$. It is easy to see that $G(F)=E^{c}\times \{x_{0}\}\bigcup\{(t,x)\in [0,1]\times S(X):f(t)(x)>0\}$. 

We next show that the function $(t,x)\to f(t)(x)$ is measurable. Indeed, let $f$ is an indicator function, say, $f = x_0^{***} \chi_A$ for some measurable set $A$ and $x_{0}^{***}\in S(X^{***})$. Then
\[
f(t)(x) =
\begin{cases} 
x_{0}^{***}(x), & t \in A, \\
0, & t \notin A.
\end{cases}
\]
As $x_{0}^{***}$ is continuous function on $S(X)$, we get that the function $(t,x)\to f(t)(x)$ is measurable, when $f$ is an indicator function. Since a simple function is a finite sum of indicator functions, we have that the function $(t,x)\to f(t)(x)$ is measurable, when $f$ is a simple function (see \cite[Chapter~II]{DU}). In the case, when $f$ is not a simple function. Since $f:[0,1]\to X^{***}$ is a Bochner $p$-integrable, we have a sequence $(s_{n})_{n\in \mathbb{N}}\in L^{q}([0,1], X^{***})$ of simple functions such that $s_{n}(t)\to f(t)$ almost everywhere in the norm. In particular, $s_{n}(t)(x)\to f(t)(x)$ almost everywhere and for each $x\in S(X)$. We know that for each $n\in \mathbb{N}$, the function $(t,x)\to s_{n}(t)(x)$ is measurable. This gives that $f$ is measurable as it is the limit of a sequence of measurable functions (see \cite[Chapter~II]{DU}). 

Consequently, $G(F)$ is a Borel set. Since $X^*$ is separable, so $X$ is separable. We have $S(X)$ is a Polish space (separable completely metrizable topological space) with respect to the norm topology. Applying Von-Neumann's selection theorem to $G(F)$, we get a measurable selector $h: [0,1]\to S(X)$ such that $f(t)(h(t))>0$ for each $t\in E$. Since $X$ is separable, by Pettis measurability theorem (See \cite[Pg.~42]{DU}), we have that the function $h:[0,1]\to S(X)$ is strongly measurable. That is, $h\in L^{p}([0,1],X)$. As we have choosen $f\in L^{p}([0,1], X)^{\perp}$, we must have
\[
f(h)=\int_{[0,1]}f(t)(h(t))dt=\int_{E}f(t)(h(t))dt=0
\]
So we get, $f(t)(h(t))=0$ almost everywhere on $E$. This is a contradiction, as $f(t)(h(t))>0$ for each $t\in E$. This establishes the claim that $L^{p}([0,1], X)^{\perp}\subseteq L^{q}([0,1], X^{\perp})$. By the hypothesis, $X$ is coreflexive which implies $X^{\perp}$ is reflexive. This gives that $L^{q}([0,1], X^{\perp})$ is reflexive. So we have that $L^{p}([0,1], X)^{\perp}$ is a reflexive subspace of $L^{p}([0,1], X)^{***}$. Thus by Proposition~\ref{P27}, the space $L^{p}([0,1], X)$ is coreflexive. This completes the proof.
\end{proof}

Recall from the introduction that a space $X$ is quasi-reflexive if the quotient space $X^{**}/X$ is finite dimensional. It is easy to see that given a quasi-reflexive space $X$, there exists a continuous projection $P$ from $X^{**}$ on to $X$ and a finite dimensional subspace $F\subseteq X^{**}$ such that ker$(P)=F$ and range$(P)=X$. So we can write $X^{**}=X\oplus F$. It is straightforward to check from the last decomposition that $X$ is a dual space. Since $X$ is a dual space, it admits a canonical projection from $X^{**}$ onto $X$. Therefore, we may assume that the projection $P$ mentioned above is contractive. We use this projection $P$ to establish Lemma~\ref{T214}. Let us begin by recalling a well know result established by \textbf{Zhibao Hu and Bor-Luh Lin} (See \cite[Proposition~2.9]{HUL}), that for any closed convex subset $K\subseteq X^*$, we have that $K$ and $\overline{K}^{\textnormal{w}^*}$ have same w$^*$-PC's. The following fact is a consequence of the above proposition.
\begin{fact}\label{F38}
Let $X$ Banach space. If $x\in X_{1}$ is a \textnormal{w}-PC, then it remains a \textnormal{w}-PC in all the higher even-order dual unit balls.
\end{fact}
\begin{proof}
	To see that $x\in X_{1}^{**}$ is a w-PC, we apply the above proposition to $K=X_{1}\subseteq X_{1}^{**}$. By using this proposition inductively to $K=X_{1}^{(n)}$ for each even $n\in \mathbb{N}$, one can show that $x$ remains a w-PC in all the higher even-order dual unit balls.
\end{proof}
In the below, we provide an easy alternative proof of Fact~\ref{F38} for a special class of Banach spaces, namely, quasi-reflexive spaces.

\begin{lem}\label{T214}
	Let $X$ be a quasi-reflexive space. If $x\in X_{1}$ is a \textnormal{w}-PC, then it remains a \textnormal{w}-PC in all the higher even-order dual unit balls.
\end{lem}
\begin{proof}
	Since $X$ is quasi-reflexive, we have $X^{**}=X\oplus F$ for some finite dimensional subspace $F$ of $X^{**}$. Let $P:X^{**}\to X^{**}$ be the corresponding contractive projection such that ker$(P)=F$ and range$(P)=X$ and let $x\in X_{1}$ be a w-PC. We first show that it remains a w-PC in $X^{**}_{1}$ by showing that any net in $X_{1}^{**}$ converging to $x$ in the w-topology, converges in the norm. Let $\{x_{\alpha}^{**}\}_{\alpha\in J}\subseteq X_{1}^{**}$ be a net such that $x_{\alpha}^{**}\to x$ in the w-topology. Since $P$ is weakly continuous, we have that $P(x_{\alpha}^{**})\to x$ in the w-topology, so in the norm. Note that the net $\{x_{\alpha}^{**}-P(x_{\alpha}^{**})\}_{\alpha\in J}\subseteq F$ and converges to $0$ in the w-topology. Since $F$ is finite dimensional, we have $x_{\alpha}^{**}-P(x_{\alpha}^{**})\to 0$ in norm. By adding the last two norm-convergences we get $x_{\alpha}^{**}\to x$ in the norm. A similar line of reasoning shows that $x$ retains the w-PC property in every higher even-order dual unit ball.
\end{proof}

The following note outlines several open problems that are motivated by our earlier results.

\begin{Q}
	Conclusion of Theorem~\ref{T25} is unknown without assuming that $X$ is \textnormal{w}$^*$-sequently dense in $X^{**}$.
\end{Q}
\begin{Q}
	It is not known whether the conclusion of Theorem~\ref{T210} holds without assuming either that both $X^*$ and $X^{**}$ satisfy the R.N.P., or that $X^{**}/X$ is a separable space. 
\end{Q}
\begin{Q}
	It is not known whether the space $L^{p}(\mu,X)$ is \textnormal{w}$^*$-sequently dense in its bidual whenever $X$ is \textnormal{w}$^*$-sequently dense in its bidual. Even in the case, where both $X^*$ and $X^{**}$ have the R.N.P., This question remains open.
\end{Q}
\section*{Acknowledgements}
The author wishes to express sincere gratitude to Prof. T. S. S. R. K. Rao and Dr. Priyanka Grover for their exceptional guidance, unwavering support, and inspiring mentorship throughout this work. Their profound insights, thoughtful feedback, and constant encouragement have been instrumental at every stage of this research. The author feels truly privileged to have had the opportunity to learn under their supervision.

The author declares that there is no conflict of interest.
\bibliographystyle{plain, abbrv}

\end{document}